\documentclass[11pt]{amsart}

\theoremstyle{plain}
\newtheorem{thm}{Theorem}[section]
\newtheorem{theorem}[thm]{Theorem}

\newtheorem{lemma}[thm]{Lemma}

\newtheorem{proposition}[thm]{Proposition}
\theoremstyle{definition}
\newtheorem{remark}[thm]{Remark}

\newtheorem{definition}[thm]{Definition}

\newtheorem{example}[thm]{Example}

\newtheorem{question}[thm]{Question}

\numberwithin{equation}{section}

\newcommand{\sC}{{\mathcal C}}

\newcommand{\sI}{{\mathcal I}}

\newcommand{\sK}{{\mathcal K}}
\newcommand{\sL}{{\mathcal L}}
\newcommand{\sM}{{\mathcal M}}

\newcommand{\sO}{{\mathcal O}}
\newcommand{\sP}{{\mathcal P}}

\newcommand{\sR}{{\mathcal R}}

\newcommand{\sU}{{\mathcal U}}

\newcommand{\sW}{{\mathcal W}}

\newcommand{\C}{{\mathbb C}}

\newcommand{\BP}{{\mathbb P}}

\title[Families of holomorphic embeddings]{Extending Nirenberg-Spencer's question on holomorphic embeddings  to families of holomorphic embeddings}
\author[Jun-Muk Hwang]{Jun-Muk Hwang} 
\address{Institute for Basic Science, Center for Complex Geometry, Daejeon, 34126, Republic of Korea}
\email{jmhwang@ibs.re.kr}

\begin{document}

\begin{abstract} Nirenberg and Spencer posed the question whether the germ of a compact complex submanifold  in a  complex manifold  is determined by its infinitesimal neighborhood of finite order when the normal bundle is sufficiently positive. To study the problem for  a larger class of submanifolds, including free rational curves,  we reformulate the question in the setting of families of submanifolds and their infinitesimal neighborhoods.  When the submanifolds have no nonzero vector fields, we prove that it is sufficient to consider only first-order neighborhoods to have an affirmative answer to the reformulated question. When the submanifolds do have nonzero vector fields, we obtain an affirmative answer to the question under the additional assumption that submanifolds have certain nice deformation properties, which  is applicable to free rational curves.    As an application, we obtain a stronger version of the Cartan-Fubini type extension theorem for Fano manifolds of Picard number 1. We also propose a potential application on hyperplane sections of projective  K3 surfaces.
\end{abstract}

\maketitle

\noindent {\sc Keywords.} infinitesimal neighborhood,  Cartan's equivalence method, free rational curves, K3-surfaces

\noindent {\sc MSC2010 Classification.} 32C22, 58A15, 14J45, 14J28

\section{Introduction}
We work in the complex-analytic setting. All geometric objects are holomorphic and open subsets refer to Euclidean topology.  The holomorphic tangent bundle of a complex manifold $X$ is denoted by $T_X$.
When $Y$ is a closed complex submanifold in a complex manifold $X$, its normal bundle $T_X|_Y /T_Y$ is denoted by $N_{Y/X}$.
For  a nonnegative integer $\ell,$ denote by $(Y/X)_{\ell}$
the $\ell$-th infinitesimal neighborhood of $Y$ in $X$, i.e., the analytic space defined
  by the $(\ell+1)$-th power of the ideal of $Y \subset X$.  The germ of (Euclidean) neighborhoods of $Y$ in $X$ is denoted by $(Y/X)_{\sO}.$

  The problems we discuss have originated from the following question  on holomorphic embeddings asked by Nirenberg and Spencer on p. 135 of  \cite{NS}.

   \begin{question}\label{q.NS}
    Let $A \subset X$ be a compact submanifold in a complex manifold whose normal bundle is positive in a suitable sense.  Is
there a positive integer $\ell$ such that  for any submanifold   $ \widetilde{A} \subset \widetilde{X}$ admitting a biholomorphic map of complex spaces  $$(A/X)_{\ell} \cong (\widetilde{A}/\widetilde{X})_{\ell},$$  we have a biholomorphic map of germs  $$(A/X)_{\sO} \cong (\widetilde{A}/\widetilde{X})_{\sO}?$$
       \end{question}

    This question and its formal version (i.e. when $\ell = +\infty$) have been studied by several authors
    (see \cite{Gr66}, \cite{MR} and references therein).
  Among others, Nirenberg and Spencer gave a positive answer under an additional  geometric assumption (having a transversely foliated neighborhood) on $(A/X)_{\sO}$ and   Griffiths gave an affirmative answer in Theorem II in p. 379 of \cite{Gr66}
      when $\dim A \geq 3$ and the normal bundle $N_{A/X}$ is sufficiently positive.
      In these works, the positivity conditions required for $N_{A/X}$ were rather strong,
limiting the applicability of the results.
 To overcome this limitation, there have been  attempts to replace the condition of the positivity of $N_{A/X}$ by more geometric conditions in terms of deformations of $A$ in $X$. The relevance of deformations of submanifolds in this context was suggested  already  in Griffiths' work (Section III.3 of \cite{Gr66}) and  later Hirschowitz in \cite{Hir} obtained results on the formal version of Question \ref{q.NS} under assumptions on the deformations. Ignoring some technicalities, we can describe the class of submanifolds considered by Hirschowitz (e.g. in Conjecture on page 501 of \cite{Hir}) as follows.

   \begin{definition}\label{d.free}
    A compact submanifold
    $A \subset X$ of a complex manifold $X$ is {\em free} if $H^0(A, N_{A/X})$ generates $N_{A/X}$ and all elements of $H^0(A, N_{A/X})$ are unobstructed in $X$. If we  denote by $$ {\rm Douady}(X) \stackrel{\rho}{\leftarrow} {\rm Univ}(X) \stackrel{\mu}{\rightarrow} X$$
    the universal family morphisms of the Douady space parametrizing compact analytic subspaces of $X$ (i.e., Hilbert scheme when $X$  is algebraic), then a compact submanifold $A \subset X$ is free if and only if ${\rm Douady}(X)$ is nonsingular at the corresponding point  $[A] \in {\rm Douady}(X)$
      and the holomorphic map  $\mu: {\rm Univ}(X) \to X$ is submersive in a neighborhood of $\rho^{-1}([A])$.
   \end{definition}

\begin{example}\label{e.rational}
When $A = \BP^1 \subset X$ is a nonsingular rational curve, it is free if and only if $N_{A/X}$ is semi-positive. This agrees with the usual definition of  free rational curves (e.g. in Section 1.1 of \cite{HM98}). Free rational curves arise naturally in the study of uniruled projective manifolds (e.g., see \cite{HM98}). \end{example}

 Let us consider the following more precise version of Question \ref{q.NS}.

 \begin{question}\label{q.free}
 Let $A \subset X$ be a  free submanifold.
 Is
there a positive integer $\ell$ such that for any free submanifold   $ \widetilde{A} \subset \widetilde{X}$ admitting a biholomorphic map of complex spaces  $$(A/X)_{\ell} \cong (\widetilde{A}/\widetilde{X})_{\ell},$$ we have a biholomorphic map of germs  $$(A/X)_{\sO} \cong (\widetilde{A}/\widetilde{X})_{\sO}?$$
 \end{question}

Unfortunately, there are simple counterexamples to Question \ref{q.free},  e.g.  Example \ref{e.blowup} below.
     In Example \ref{e.blowup}, the submanifold $\widetilde{A} \subset \widetilde{X}$ used to give a counterexample to Question  \ref{q.free}  is not general among its deformations.   To avoid such negative answers,  we have to consider finite-order neighborhoods of all deformations simultaneously and ask the question for general members of the family.
Namely, we ask the following {\em family-version} of Questions \ref{q.NS} and \ref{q.free}
(see Remark \ref{r.q} after Definition \ref{d.iso} for a more precise formulation).

 \begin{question}\label{q.family}
 Let $\sK \subset {\rm Douady}(X)$ be a connected open subset (in Euclidean topology) whose members are free submanifolds in $X$.
 Is
there a positive integer $\ell$ such that for any connected open subset   $ \widetilde{\sK} \subset {\rm Douady}(\widetilde{X})$ for which there exists a biholomorphic map $f:\sK \to \widetilde{\sK}$ and a biholomorphic map of complex spaces  $$(A/X)_{\ell} \cong (\widetilde{A}/\widetilde{X})_{\ell}$$ for each $[A] \in \sK$ and $[\widetilde{A}] = f([A]),$  we have a biholomorphic map of germs  $$(A/X)_{\sO} \cong (\widetilde{A}/\widetilde{X})_{\sO}$$ for some member $[A] \in \sK$?
 \end{question}

 We have the following two positive results on Question \ref{q.family}.

  \begin{theorem}\label{t.2}
 In Question \ref{q.family}, assume that $H^0(A, T_A) =0$ for a  general $[A] \in \sK$. Then  $\ell = 1$ has the required property.
  \end{theorem}

\begin{theorem}\label{t.1}
In Question \ref{q.family}, assume that the normal bundle $N_{A/X}$ of a general member $A \subset X$ of $\sK$ is separating, in the sense that   $$H^0(A, N_{A/X} \otimes {\bf m}_x) \neq H^0(A, N_{A/X} \otimes {\bf m}_y)$$ for any $x \neq y \in A$, where ${\bf m}_x$ denotes the maximal ideal of $x$. Then there exists a positive integer $\ell = \ell (\sK)$ with the required property. \end{theorem}

The assumption of a family of isomorphisms $(A/X)_{\ell} \cong (\widetilde{A}/\widetilde{X})_{\ell}$ for all $[A] \in \sK$ in these two theorems is certainly a strong requirement. Luckily, there are many situations where this stronger condition can be checked effectively,  some of which we exhibit below.

The proofs of Theorem \ref{t.2} and Theorem \ref{t.1}
are of different nature. Whereas the proof of Theorem \ref{t.2} is essentially algebraic,
the proof of  Theorem \ref{t.1} uses tools from differential geometry, the theory of systems of analytic partial differential equations. In particular, it uses Morimoto's work \cite{Mor} on Cartan's equivalence method.
Ultimately, the bound $\ell$ in Theorem \ref{t.1} arises from the Hilbert basis theorem and the effective results on Spencer's  $\delta$-Poincare estimate employed in \cite{Mor}. We mention that the formal aspect of \cite{Mor} was used  in \cite{Hw19}  to study a formal version of Question \ref{q.NS}.

  Unlike Theorem \ref{t.2}, the number $\ell= \ell(\sK)$ in Theorem \ref{t.1} depends on the family $\sK$. In fact, for each positive integer $k$, there is an example, Example \ref{e.t.1} below, of $\sK$ with $\ell(\sK) \geq k$ (see, however, Remark \ref{r.JeVe}). It is a challenging problem to compute $\ell$ explicitly for concrete examples of $X$ and $\sK$ in Theorem \ref{t.1}, which we leave for future studies. Let us just mention that in the case of $X = \BP^n$ with the family $\sK$ consisting of all lines on $\BP^n$, one can deduce $\ell=3$ from the results on generalized path geometries in p. 462 of \cite{CS}.

The most interesting case of Theorem \ref{t.1} is when $\sK$ is a family of free rational curves. For free rational curves, we can always assume the separating condition on the normal bundle in Theorem \ref{t.1}.  In fact, if the normal bundle is not separating, it is a trivial bundle, in which case the rational curve admits a holomorphic tubular neighborhood, i.e. a neighborhood biholomorphic to the product of the rational curve with a complex ball.    Of particular interest is when the free rational curves come from minimal rational curves (in the sense of \cite{HM98}). As a matter of fact, the original motivation of Theorem \ref{t.1} was to use it to prove the following version of Cartan-Fubini type extension theorem (see \cite{HM01} and Theorem \ref{t.CFO} below) for minimal rational curves on Fano manifolds of Picard number 1.

\begin{theorem}\label{t.CF}
Let $X$ be Fano manifold of Picard number 1.
 Let $\sM$  be a family of minimal rational curves on $X$, i.e., an irreducible component of the space of rational curves on $X$  such that  for a general point $x \in X$,
the subscheme $\sM_x \subset \sM$  consisting of members through $x$  is projective and nonempty. Assume \begin{itemize}
  \item[(i)] the subscheme $\sM_x$ is irreducible for a general $x \in X$; and
      \item[(ii)]  general members of $\sM$ are nonsingular. \end{itemize}
          Then there exists a positive integer $\ell$ such that for any Fano manifold $\widetilde{X}$ of Picard number 1 and a family $\widetilde{M}$ of minimal rational curves on $\widetilde{X}$ with irreducible $\widetilde{\sM}_{\widetilde{x}}$ for a general $\widetilde{x} \in \widetilde{X}$,  if there are connected nonempty open subsets $\sK \subset \sM$ and $\widetilde{\sK} \subset \widetilde{\sM}$ with a biholomorphic map $f:\sK \to \widetilde{\sK}$ and a biholomorphic map of complex spaces  $$(A/X)_{\ell} \cong (\widetilde{A}/\widetilde{X})_{\ell}$$ for each $[A] \in \sK$ and $[\widetilde{A}] = f([A]),$ then there exists a biregular morphism $ X \to \widetilde{X}$ which induces $f$. \end{theorem}

When compared with the results of \cite{HM01} (or Theorem \ref{t.CFO} below), the advantage of Theorem \ref{t.CF} is that it can be stated in purely algebraic terms (e.g. by replacing $\sK$ by an etale open subset of $\sM$).

The assumptions in Theorem \ref{t.CF} are satisfied by many examples of Fano manifolds of Picard number 1: rational homogeneous spaces, smooth complete intersections with index at least 3 and moduli spaces of stable bundles with coprime degree and rank on curves (see the examples  in Section 1.4 of \cite{Hw01}).
Let us  mention that the assumption (i) in Theorem \ref{t.CF} is necessary. For example, the theorem fails  when $X$ is a cubic threefold and $\sK$ is the family of lines on $X$. On the other hand, the assumption (ii) is  not essential. In fact, one can reformulate Theorem \ref{t.CF} in terms of the graphs of the normalization of the rational curves if general members of $\sM$ are singular, as in Theorem 1.9 of \cite{Hw19}.

As an application of Theorem \ref{t.2}, we look into a question on the geometry of projective K3 surfaces.
Let $X \subset \BP^g$ be a linearly normal projective K3 surface. In \cite{CLM}, Ciliberto, Lopez and Miranda  proved that if the hyperplane bundle generates ${\rm Pic}(X)$ and $g \geq 13$ (or $g=11$), then the isomorphism type of a general hyperplane section $A$ determines $X$, i.e., any other K3-surface $\widetilde{X} \subset \BP^g$ which has a hyperplane section $\widetilde{A}$ biregular to $A$ and generating  ${\rm Pic}(\widetilde{X})$ must be projectively isomorphic to $X$.
More refined results have been obtained recently in \cite{CDS}, but the condition $g \geq 13$  and some  restrictions on ${\rm Pic}(X)$ are necessary.
Without such restrictions, the following seems a reasonable question.

   \begin{question}\label{q.K3}
   Let $X, \widetilde{X} \subset \BP^g, g>2, $ be two linearly normal projective K3 surfaces.
Assume that the sets of isomorphism classes of general hyperplane sections of $X$ and $\widetilde{X}$
  coincide as subsets of the moduli space $\sM_g$ of curves of genus $g > 2$. Are $X$ and $\widetilde{X}$ projectively isomorphic? \end{question}

As a potential approach to Question \ref{q.K3}, we explain in Section 4 how  Theorem \ref{t.2} can be used to reduce it to certain questions on canonical rings of nonhyperelliptic curves.

\bigskip
   The paper is organized as follows. In Section 2, we introduce the notion of the iso-equivalence up to finite order, which is a precise formulation of the conditions in Question \ref{q.family},  and give some immediate examples.  The proof of Theorem \ref{t.2} is in Section 3 and the discussion of its potential application to Question \ref{q.K3} is in Section 4. Only algebro-geometric arguments are employed up to Section 4. The rest of the paper uses differential geometric tools.  In Section 5, further examples are given that arise from submanifolds of projective space having contact of finite order. Section 6 presents the theory of G-structures and proves a key technical result, Theorem \ref{t.G}, using  Morimoto's work \cite{Mor}. In Section 7, we prove Theorem \ref{t.1} using the result of Section 6 and prove also Theorem \ref{t.CF}.

\section{Iso-equivalence of families of free submanifolds and some immediate examples}

We introduce the following  notions, which makes the condition in Theorem \ref{t.1} more systematic.

\begin{definition}\label{d.iso} For a complex manifold $X$, a connected open subset $\sK \subset {\rm Douady}(X)$ is called a {\em free family} if its members are free compact submanifolds of $X$.  Let $\sK \subset {\rm Douady}(X)$ and $\widetilde{\sK} \subset {\rm Douday}(\widetilde{X})$  be   free families  in complex manifolds $X$ and $\widetilde{X}$ of the same dimension.  Let $$\sK \stackrel{\rho}{\leftarrow} \sU \stackrel{\mu}{\rightarrow} X \ \mbox{ and } \widetilde{\sK} \stackrel{\widetilde{\rho}}{\leftarrow} \widetilde{\sU} \stackrel{\widetilde{\mu}}{\rightarrow} \widetilde{X}$$ be the universal families restricted to $\sK$ and $\widetilde{\sK}$.  \begin{itemize} \item[(1)] For each nonnegative integer $k$, denote by $$\sK \stackrel{\rho_{k}}{\leftarrow} \sU_{k} \stackrel{\mu_{k}}{\rightarrow} X$$ the family whose fiber at $[A] \in \sK$ is the $k$-th infinitesimal neighborhood $(A/X)_k$. More precisely, viewing $\sU$ as a closed submanifold in the complex manifold $\sK \times X$, we can define $\sU_k$ as the $k$-th infinitesimal neighborhood  $(\sU/\sK \times X)_k$.  We have the  natural inclusion $j^{\ell}_k$  for $0 \leq \ell <k$:
 $$ \begin{array}{lllll}
\sK & \stackrel{\rho_{\ell}}{\leftarrow} & \sU_{\ell} &  \stackrel{\mu_{\ell}}{\rightarrow}  & X\\
\parallel & & \downarrow j^{\ell}_k  & & \parallel \\
\sK & \stackrel{\rho_{k}}{\leftarrow} & \sU_{k} &\stackrel{\mu_{k}}{\rightarrow} & X
 \end{array} $$ We define $\widetilde{\sU}_k$ and $\widetilde{j}^{\ell}_k$ similarly.
  \item[(2)] For a nonnegative integer $k$, we say that $\sK$ and $\widetilde{\sK}$ are {\em iso-equivalent up to order} $k$,
if there are connected
 open subsets $\sW \subset \sK$ and $\widetilde{\sW} \subset \widetilde{\sK}$ equipped with biholomorphic maps $f: \sW \to \widetilde{\sW}$ and
 $F_{\ell}: \rho_{\ell}^{-1}(\sW) \to \widetilde{\rho}_{\ell}^{-1}(\widetilde{\sW})$ for each $0 \leq \ell \leq k$ such that $$f \circ \rho_{\ell}|_{\rho_{\ell}^{-1}(\sW)} =   \widetilde{\rho}_{\ell} \circ F_{\ell} \mbox{ and }
F_k \circ j^{\ell}_k|_{\rho_{\ell}^{-1}(\sW)} = \widetilde{j}^{\ell}_k \circ F_{\ell}.$$
\item[(3)] We say that $\sK$ and $\widetilde{\sK}$ are {\em germ-equivalent} if there exist some members $[A] \in \sK$ and $[\widetilde{A}] \in \widetilde{\sK}$ such that $(A/X)_{\sO} \cong (\widetilde{A}/\widetilde{X})_{\sO}$.  \end{itemize} \end{definition}

  The following is straight-forward from the definition.

    \begin{lemma}\label{l.iso}
     Let $\sK$ and $\widetilde{\sK}$ be as in Definition \ref{d.iso}.
      \begin{itemize} \item[(1)] If $\sK$ and $\widetilde{\sK}$ are iso-equivalent up to order $k$, then they are iso-equivalent up to order $\ell$ for any $0 \leq \ell \leq k$.
      \item[(2)] If $\sK$ and $\widetilde{\sK}$ are germ-equivalent, then they are iso-equivalent up to order $k$ for any nonnegative integer $k$. \end{itemize} \end{lemma}

\begin{remark}\label{r.q}
Using the terminology of Definition \ref{d.iso}, we can restate Question \ref{q.family} as follows. \begin{itemize} \item  For a free family $\sK \subset {\rm Douady}(X)$, is there a positive integer $\ell$ such that if $\sK$ and a free family $\widetilde{\sK} \subset {\rm Douady}(\widetilde{X})$ are iso-equivalent up to order $\ell,$ then they are germ-equivalent? \end{itemize} \end{remark}

To give some immediate nontrivial examples of Definition \ref{d.iso}, let us recall the following definition.

\begin{definition}\label{d.unbendable}
A free nonsingular rational curve $A$ on a complex manifold $X$ is {\em unbendable} if its normal bundle is of the form
$$N_{A/X} \cong \sO(1)^{\oplus p} \oplus \sO^{\oplus (n -p-1)}$$ where $n= \dim X$ and  $p$ is some nonnegative integer less than $n$. Note that $p+2 = A \cdot K_X^{-1}$ is the {\em anti-canonical degree} of $A$. \end{definition}

Unbendable rational curves arise naturally in the study of uniruled projective manifolds, as general members of minimal rational curves (e.g., see \cite{HM98} for a survey, where  unbendable rational curves were called standard rational curves).
It is easy to check when families of unbendable rational curves are iso-equivalent up to order 1:

 \begin{proposition}\label{p.mrc}
Let $\sK \subset {\rm Douady}(X)$ (resp. $\widetilde{\sK} \subset {\rm Douady}(\widetilde{X})$) be a free family of unbendable rational curves    on a complex manifold $X$ (resp. $\widetilde{X}$).  Then $\sK$ and $\widetilde{\sK}$ are iso-equivalent up to order 1 if and only if the anti-canonical degrees of their members are equal. \end{proposition}

\begin{proof}
  It is well-known (see the remark after Proposition 1.7 in \cite{Gr66} or Section 4.1 of \cite{CDS}) that  the first infinitesimal neighborhood of a submanifold $A \subset X$ is determined by the extension class of
$$ 0 \to T_A \to T_X|_A \to N_{A/X} \to 0.$$
When $A$ is an unbendable rational curve, we have $T_A \cong \sO(2)$ and $N_{A/X} \cong \sO(1)^{\oplus p} \oplus \sO^{\oplus (n-1-p)}, n = \dim X$, which implies  $$H^1(A, T_A \otimes N^*_{A/X}) = H^1(\BP^1, \sO(1)^{\oplus p} \oplus \sO(2)^{\oplus (n-1-p)}) =0.$$ Thus the exact sequence always splits and the first infinitesimal neighborhood is uniquely determined by $p$. \end{proof}

 Proposition \ref{p.mrc} shows that   the iso-equivalence up to order 1 is a weak condition for families of rational curves, in contrast to Theorem \ref{t.2}.
We can go one step higher in the following case.

\begin{proposition}\label{p.quasi}
An unbendable rational curve $A$ on a complex manifold $S$ is called a quasi-line if $N_{A/S} \cong \sO(1)^{\oplus (n -1)}, n = \dim S.$ Two free families $\sK \subset {\rm Douady}(X)$ and $\widetilde{\sK} \subset {\rm Douady}(\widetilde{X})$ of quasi-lines on complex manifolds $X$ and $\widetilde{X}$ of the same dimension $n$ are iso-equivalent up to order 2. \end{proposition}

\begin{proof}
Let $A\subset X = \BP^n$ be a line. For a quasi-line $\widetilde{A} \subset \widetilde{X}$, the  obstruction to lift the equivalence $(A/X)_1 \cong (\widetilde{A}/ \widetilde{X})_1$ of Proposition \ref{p.mrc}  to the second infinitesimal neighborhoods lies in
 $$H^1( \widetilde{A}, T_{\widetilde{X}}|_{\widetilde{A}} \otimes {\rm Sym}^2 N^*_{\widetilde{A}/\widetilde{X}})$$ by Theorem 2.5 of \cite{MR}. This cohomology group vanishes by $N_{\widetilde{A}/\widetilde{X}} \cong \sO(1)^{\oplus (n -1)}$ and
 $T_X|_A \cong \sO(2) \oplus \sO(1)^{\oplus (n-1)}.$ \end{proof}

\section{Proof of  Theorem \ref{t.2}}

We prove the following  theorem which restates Theorem \ref{t.2} in terms of Definition \ref{d.iso}.

\begin{theorem}\label{t.2g}
Let $\sK \subset {\rm Douady}(X)$ be a  free family  in a complex manifold $X$ containing  a member $A \subset X$ which satisfies $H^0(A, T_A) =0$. Then for any free family $\widetilde{\sK} \subset {\rm Douady}(\widetilde{X})$  in a  complex manifold $\widetilde{X}$, if $\sK$ and $\widetilde{\sK}$ are  iso-equivalent up to order $1$, then they are germ-equivalent. \end{theorem}

 To prove this, we use the condition $H^0(A, T_A) =0$ via the following  lemma, which is contained in Theorem 4.1 of \cite{Pa} and also in the proof of Corollary 18.3 in \cite{Ha10}.

\begin{lemma}\label{l.Har}
Let $\pi: U \to Y$ be a proper smooth morphism between two complex manifolds and
let $B$ be the fiber $\pi^{-1}(y)$ over a point $y \in Y$. If $H^0(B, T_B) =0$, then an automorphism of the first infinitesimal neighborhood  $\varphi: (B/U)_1 \to (B/U)_1$ satisfying $\pi \circ \varphi = \pi|_{(B/U)_1}$ and $\varphi|_B = {\rm Id}_B$ must be the identity ${\rm Id}_{(B/U)_1}$. \end{lemma}

\begin{proof}
This is immediate from  Theorem 4.1 in \cite{Pa}. More precisely,
let ${\rm Aut}_y((B/U)_1)$ be the group of all automorphisms
 $\varphi: (B/U)_1 \to (B/U)_1$ satisfying $\pi \circ \varphi = \pi|_{(B/U)_1}$ and $\varphi|_B = {\rm Id}_B.$
 Since $(B/U)_1$ is equal to the fiber product   of $\pi$ and the inclusion  $(y/Y)_1 \subset Y$,   the exact sequence of pointed spaces in Theorem 4.1 of \cite{Pa} applied to the inclusion $y \in (y/Y)_1$ reads
$$ 0 \to H^0(B, T_B) \otimes I \to {\rm Aut}_y((B/U)_1) \to \{ {\rm Id}_B \} \to \cdots,$$ where $I$ is the  ideal of $y$ in $(y/Y)_1$. Thus  ${\rm Aut}_y((B/U)_1) = \{{\rm Id}_B \}$ follows from $H^0(B, T_B) =0$.
 \end{proof}

The next lemma is  a reformulation of Propositions 5.7 of \cite{Hir} in our terminology.  Its formal version is Lemma 3.5 of \cite{Hw19}.

\begin{lemma}\label{l.Hir}
 Let $\sK \stackrel{\rho}{\leftarrow} \sU \stackrel{\mu}{\to} X $ and $\widetilde{\sK} \stackrel{\widetilde{\rho}}{\leftarrow} \widetilde{\sU} \stackrel{\widetilde{\mu}}{\to} \widetilde{X}$ be two free families. For $[A] \in \sK$ and $[\widetilde{A}] \in \widetilde{\sK}$,  an isomorphism $\varphi: (A/X)_k \to (\widetilde{A}/\widetilde{X})_k$ induces natural isomorphisms of complex spaces $$\varphi^{\sK}: ([A]/\sK)_k \to ([\widetilde{A}]/\widetilde{\sK})_k$$ and
$$\varphi^{\sU}: (\rho^{-1}([A])/ \sU)_k \to (\widetilde{\rho}^{-1}([\widetilde{A}])/\widetilde{\sU})_k$$
which commute with the restrictions of $\rho, \mu$  on $(\rho^{-1}([A])/ \sU)_k$ and  $\widetilde{\rho}, \widetilde{\mu}$ on $(\widetilde{\rho}^{-1}([\widetilde{A}])/\widetilde{\sU})_k$, respectively. \end{lemma}

 For the reader's convenience, we provide the proof of the following elementary fact.

\begin{lemma}\label{l.tangent}
Let $B \subset U$ and $A \subset X$ be compact submanifolds of complex manifolds.
  A morphism of complex spaces $\psi: (B/U)_1 \to (A/X)_1$ induces, for each point $u \in B$ and its set-theoretic image $x = \psi(u) \in A$, a natural homomorphism ${\rm d}_u \psi: T_{U,u} \to T_{X, x},$ which coincides with the usual derivative of a map when $\psi$ comes from a holomorphic map $(B/U)_{\sO} \to (A/X)_{\sO}$. \end{lemma}

\begin{proof}
The homomorphism of the stalks associated to $\psi$
$$\psi^{\#}: \sO_{X,x}/\sI_{A,x}^2 \to \sO_{U,u}/\sI_{B,u}^2$$ induces a homomorphism
${\bf m}_{X,x}/\sI_{A,x}^2 \to {\bf m}_{U,u}/\sI_{B,u}^2$ for the maximal ideals
${\bf m}_{X,x} \subset \sO_{X,x}$ and ${\bf m}_{U,u} \subset \sO_{U,u}$.
 By the inclusions $\sI_{A,x}^2 \subset {\bf m}_{X,x}^2$ and $\sI_{B,u}^2 \subset {\bf m}_{U,u}^2,$ we have the induced homomorphism
 ${\bf m}_{X,x}/{\bf m}_{X,x}^2 \to {\bf m}_{U,u}/{\bf m}_{U,u}^2$ whose dual defines the homomorphism ${\rm d}_u \psi$. \end{proof}

\begin{proof}[Proof of Theorem \ref{t.2g}]
Assume that $\sK$ and $\widetilde{\sK}$ are free families that are iso-equivalent up to order $1$.
Shrinking $\sK$ and $\widetilde{\sK}$ if necessary, we may assume that $H^0(A, T_A) =0$ for every $[A] \in \sK$ and we have  biholomorphic maps $$f : \sK \to \widetilde{\sK}, \  F_0: \sU \to \widetilde{\sU} \mbox{ and }F_1: \sU_1 \to \widetilde{\sU}_{1}$$ satisfying the properties of Definition \ref{d.iso} (2).

For each $[A]\in \sK$, the map $F_1$ defines an isomorphism $$ F_{1,A}: (A/X)_1 \to (\widetilde{A}/\widetilde{X})_1,$$ where $\widetilde{A} = \widetilde{\mu} \circ F_0 (\rho^{-1}([A])).$ By Lemma \ref{l.Hir}, this induces an isomorphism
$$F_{1,A}^{\sU}: (\rho^{-1}([A])/ \sU)_1 \to (\widetilde{\rho}^{-1}([\widetilde{A}])/\widetilde{\sU})_1.$$
As $F_{1,A}^{\sU}$ commutes with the restrictions  $\mu|_{(\rho^{-1}([A])/ \sU)_1}$ and $\widetilde{\mu}|_{(\widetilde{\rho}^{-1}([\widetilde{A}])/\widetilde{\sU})_1}$ as described in Lemma \ref{l.Hir}, we see that for any $u \in \rho^{-1}([A])$, the homomorphism from Lemma \ref{l.tangent} $${\rm d}_u F_{1,A}^{\sU}: T_{\sU,u}
\to T_{\widetilde{\sU}, \widetilde{u}}, \ \widetilde{u} = F_0(u)$$
sends ${\rm Ker} ({\rm d}_u \mu)$ to ${\rm Ker} ({\rm d}_{\widetilde{u}} \widetilde{\mu})$.

On the other hand, the map $F_0$ induces
$$F_0^A: (\rho^{-1}([A])/ \sU)_1 \to (\widetilde{\rho}^{-1}([\widetilde{A}])/\widetilde{\sU})_1$$ such that  the derivative induced by $F_0^A$ in Lemma \ref{l.tangent} coincides with the derivative of $F_0$  $${\rm d}_u F^A_0 = {\rm d}_u F_0: T_{\sU, u} \to T_{\widetilde{\sU}, \widetilde{u}}$$ for any $u \in \rho^{-1}([A])$ and $\widetilde{u} = F_0(u)$.
By the assumption $H^0(A, T_A) =0$ and Lemma \ref{l.Har}, we have the equality $F_0^A = F_{1,A}^{\sU}$. Thus we see that ${\rm d}_u F_0$ sends ${\rm Ker}( {\rm d}_u \mu)$ to ${\rm Ker} ({\rm d}_{\widetilde{u}} \widetilde{\mu})$  for each $u \in \sU$.
It follows that $F_0$ sends fibers of $\mu$ to fibers of $\widetilde{\mu}$. Then $F_0$ descends to $(A/X)_{\sO} \cong (\widetilde{A}/\widetilde{X})_{\sO}$ for any $[A] \in \sK$. It follows that $\sK$ and $\widetilde{\sK}$ are germ-equivalent.
\end{proof}

\section{An approach to Question \ref{q.K3} via Theorem \ref{t.2}}

In this section, we discuss an application of Theorem \ref{t.2} on Question \ref{q.K3}.
Firstly,  recall the following well-known result.

\begin{proposition}\label{p.Huy}
Let $X \subset \BP^g, g>2,$ be a linearly normal projective K3 surface. Then a general hyperplane section $A \subset X$ satisfies \begin{itemize}\item[(1)] $A$ is a nonsingular curve of genus $g$ with $N_{A/X}$ isomorphic to the canonical bundle $K_A$; \item[(2)] $A$  is not hyperelliptic; and \item[(3)] $A$ is a free submanifold of $X.$ \end{itemize} In particular, we have a free family $\sK \subset {\rm Douady}(X)$ of nonhyperelliptic  hyperplane sections  of genus $g$. \end{proposition}

\begin{proof}
(1) is an easy consequence of the adjunction formula. (2) is from
 Remark 2.4 (ii) of \cite{Huy}.  (3) follows from $$\dim H^0(A, N_{A/Y}) = \dim H^0(A, K_A) =  g.$$   \end{proof}

By Proposition \ref{p.Huy}, we can restate Question \ref{q.K3} in terms of Definition \ref{d.iso} as follows.

\begin{question}\label{q.K3g}
Let $X$ and $\widetilde{X}$ be two linearly normal projective K3 surfaces in $\BP^g, g>2.$  Suppose the free families $\sK$ and $\widetilde{\sK}$ of nonhyperelliptic hyperplane sections of $X$ and $\widetilde{X} \subset \BP^g$ are  iso-equivalent up to order 0. Are $X$ and $\widetilde{X}$ isomorphic by a projective transformation of $\BP^g$? \end{question}

As an application of Theorem \ref{t.2}, we obtain the following reduction of the problem.

\begin{proposition}\label{p.K3}
 If the free families $\sK$ and $\widetilde{\sK}$  in Question  \ref{q.K3g} are iso-equivalent up to order 1, then $X$ and $\widetilde{X}$ are isomorphic by a projective transformation of $\BP^g$ \end{proposition}

\begin{proof}
By  Theorem \ref{t.2g} and Propositions \ref{p.Huy}, we see that $\sK$ and $\widetilde{\sK}$ are germ-equivalent. So there exist Euclidean neighborhoods of some hyperplane sections $A \subset U \subset X$
and $\widetilde{A} \subset \widetilde{U} \subset \widetilde{X}$ with a biholomorphic map $\Phi: U \to \widetilde{U}$ satisfying $\Phi(A) = \widetilde{A}$. Then $\Phi$ can be extended to a birational map $\Psi$ from $X$ to $\widetilde{X}$ by Corollary V.2.3 of \cite{Ha70}. By the uniqueness of minimal models, the birational map $\Psi$ is biregular. From $\Phi(A) = \widetilde{A}$ and the linear normality of $X, \widetilde{X} \subset \BP^g$, the biregular morphism $\Psi$ comes from a projective transformation of $\BP^g$. \end{proof}

The main issue is how to check that $\sK$ and $\widetilde{\sK}$ are iso-equivalent up to order 1.
Shrinking $\sK$ and $\widetilde{\sK}$ if necessary, we can assume that there is a biholomorphic map $f: \sK \to \widetilde{\sK}$  such that for each   $[A] \in \sK$ and $[\widetilde{A}] = f([A]) \in \widetilde{\sK},$ the two curves $A$ and $\widetilde{A}$ are biholomorphic.
We can give the following partial answer to the problem.

\begin{proposition}\label{p.partial}
In the above setting, suppose the following diagram of Kodaira-Spencer maps $\kappa_A$ and $\kappa_{\widetilde{A}}$
$$\begin{array}{ccc}   H^0(A, K_A) = H^0(A, N_{A/X})  & \stackrel{\kappa_A}{\longrightarrow} &  H^1(A, T_A) \\
\| & &  \| \\ H^0(\widetilde{A}, K_{\widetilde{A}}) = H^0(\widetilde{A}, N_{\widetilde{A}/\widetilde{X}}) & \stackrel{\kappa_{\widetilde{A}}}{\longrightarrow} & H^1(\widetilde{A}, T_{\widetilde{A}})  \end{array} $$
commutes for each $A$. Then $\sK$ and $\widetilde{\sK}$ are iso-equivalent up to order 1.\end{proposition}

\begin{proof}
We need to show that the first infinitesimal neighborhoods $(A/X)_1$ and $(\widetilde{A}/\widetilde{X})_1$ are isomorphic. Recall that the first infinitesimal neighborhood of a submanifold $A \subset X$ is determined by the extension class of $$0 \to T_A \to T_X|_A \to N_{A/X} \to 0$$ up to nonzero scalar multiplications (as already mentioned in the proof of Proposition \ref{p.mrc}).
 Thus it suffices to prove that the vector bundle $T_X|_A$ and $T_{\widetilde{X}}|_{\widetilde{A}}$ are isomorphic (up to a nonzero scalar multiplication)
as extensions of  $K_A \cong K_{\widetilde{A}}$ by $T_A \cong T_{\widetilde{A}}$. This follows from the next lemma. \end{proof}

\begin{lemma}\label{l.BM}
Let $A \subset X$ be a nonhyperelliptic nonsingular hyperplane section of a linearly normal K3 surface $X \subset \BP^g$. Then the Kodaira-Spencer map $H^0(A, N_{A/X}) \to H^1(A, T_A)$ of the hyperplane sections determines the extension class of $$(\dagger) \ \ 0 \to T_A \to T_X|_A \to N_{A/X} \to 0$$ up to  nonzero scalar multiplications. \end{lemma}

\begin{proof}
 This is essentially contained in the proof of Lemma 1 of \cite{BM}.
 The boundary homomorphism $$\partial: H^0(A, K_A) = H^0(A, N_{A/X}) \to H^1(A, T_A)$$ of the exact sequence $(\dagger)$ is the Kodaira-Spencer map. Let $$\partial^*: H^0(A, K_A^{\otimes 2}) \to H^1(A, \sO_A)$$ be the dual of $\partial$.
 Let $e^*: H^0(A, K_A^{\otimes 3}) \to \C$ be the Serre dual of the extension class $e \in H^1(A, T_A^{\otimes 2})$ of $(\dagger)$.
 As explained in the proof of Lemma 1 of \cite{BM}, the following diagram, where $\alpha$ is the multiplication map and $\beta$ is the Serre duality pairing, commutes up to sign.
 $$\begin{array}{ccc}   H^0(A, K_A) \otimes H^0(A, K_A^{\otimes 2}) & \stackrel{{\rm Id} \otimes \partial^*}{\longrightarrow} & H^0(A, K_A) \otimes H^1(A, \sO_A) \\
\downarrow \alpha & &  \downarrow \beta \\ H^0(A, K_A^{\otimes 3}) & \stackrel{e^*}{\longrightarrow} & \C.  \end{array} $$  As $A$ is not hyperelliptic, the multiplication map
$\alpha$ is surjective.  Thus the Kodaira-Spencer homomorphism $\partial$ determines the extension class $e$ up to nonzero scalar multiplications by $$ \alpha^{-1}({\rm Ker}(e^*)) = {\rm Ker}(\beta \circ ({\rm Id} \otimes \partial^*)).$$
 \end{proof}

Under the assumption of Question \ref{q.K3g}, we know that the two maps $\kappa_A$ and $\kappa_{\widetilde{A}}$ have the same images in $H^1(A, T_A) = H^1(\widetilde{A}, T_{\widetilde{A}})$. But this does not imply that the diagram of Proposition \ref{p.partial} commutes. One potential approach is to answer the following question.

\begin{question}\label{q.open}
Let $A$ be a nonhyperelliptic curve (one may assume that it is a hyperplane section of a K3 surface $X \subset \BP^g$). For two elements $e, \widetilde{e} \in H^1(A, T_A^{\otimes 2})$,  let $$e^+, \widetilde{e}^+ :H^0(A, K_A ^{\otimes 2}) \to H^1(A, \sO_A)$$
be the cup product maps. If ${\rm Im}(e^+) = {\rm Im}(\widetilde{e}^+)$, is $e= \widetilde{e}$? \end{question}

The diagram in the proof of Lemma \ref{l.BM} shows that this is a question on the canonical ring of $A$.

\section{Relation with projective submanifolds having contact up to order $k$}

We recall the following notions in projective differential geometry (\cite{Gr74}, \cite{Jen}, \cite{Ve}).

\begin{definition}\label{d.contact}
Let $S$ and $\widetilde{S}$ be two (not necessarily closed) submanifolds of the same dimension $m$ in the $n$-dimensional projective space $\BP^n$.
 \begin{itemize} \item[(i)] We say that $S$ and $\widetilde{S}$ {\em agree up to order} 1 at $x \in S \cap \widetilde{S}$ if they have the same projective tangent space at $x$ in $\BP^n$. In this case, we can choose an inhomogeneous coordinate system $(w_1, \ldots, w_n)$ on $\BP^n$ centered at $x$ such that the common projective tangent space of $S$ and $\widetilde{S}$ at $x$ is given by
 $w_{m +1} = w_{m+2} = \cdots = w_n =0$ and the germ of $S$  (resp. $\widetilde{S}$) at $x$ is given by the equations
$$ \{ w_i = F_i( w_1, \ldots, w_m), \ m+1 \leq i \leq n\} $$ $$ \mbox{ ( resp. } \{ w_i = \widetilde{F}_i (w_1, \ldots, w_m), \ m+1 \leq i \leq n \} \mbox{ )}$$ for some convergent power series $F_i$ and $\widetilde{F}_i$ in  $w_1, \ldots, w_m$.
 \item[(ii)] For a positive integer $k \geq 2$, we say that $S$ and $\widetilde{S}$ {\em agree up to order} $k$ at $x \in S \cap \widetilde{S}$ if they agree up to order 1 and in terms of the coordinates in (i), the power series $F_i$ and $\widetilde{F}_i$ are equal up to terms of degree $k$.
     \item[(iii)]
 For a positive integer
 $k$, we say that $S$ and $\widetilde{S}$ have {\em contact up to order} $k$
 if there exist a biholomorphic map $h: S_o \to \widetilde{S}_o$ between nonempty connected open subsets $S_o \subset S$ and $\widetilde{S}_o \subset \widetilde{S}$,   and a holomorphic map
$g: S_o \to {\rm PGL}(n+1)$ such that for each point $s \in S_o$, the projective transformation $g(s) \in {\rm PGL}(n+1)$ sends $h(s) \in \BP^n$ to $s$ and the two submanifolds $S$ and $g(s) \cdot \widetilde{S}$ of $\BP^n$  agree up to order $k$ at $s= g(s) \cdot h(s)$. \end{itemize} \end{definition}

We recall the following theorem  (Theorem 3 in p. 32 of \cite{Jen}, Th\'eor\'eme de congruence  in p. 514 of \cite{Ve}, answering Griffiths' question in Section 5 of \cite{Gr74}).

\begin{theorem}\label{t.JeVe}
There exists a positive integer $\ell$ (depending on the dimension $n$) such that
if two  submanifolds $S, \widetilde{S} \subset \BP^n$ have  contact of order $\ell$, then  there exists $\gamma \in {\rm PGL}(n+1)$ such that $S \cap \gamma(\widetilde{S})$ contains a nonempty open subset  in $S$. \end{theorem}

We see below (Remark \ref{r.JeVe}) that (a slightly weaker version of) Theorem \ref{t.JeVe} can be interpreted as a special case of Theorem \ref{t.1}.  The  link between the two theorems is the following relation between
Definition \ref{d.contact} (ii) and  the equivalence of  infinitesimal neighborhoods of certain rational curves.

\begin{proposition}\label{p.example}
Let $S, \widetilde{S} $ be two submanifolds of the same dimension $m$ in $\BP^{n}$ which agree up to order $k$ at $x \in S \cap \widetilde{S}$.
Regarding $\BP^{n}$ as a hyperplane in $\BP^{n+1}$, let $X$ (resp. $\widetilde{X}$) be the blowup of $\BP^{n+1}$ along $S$ (resp. $\widetilde{S}$). Choose two lines $C, \widetilde{C} \subset \BP^{n+1}$ such that $x \in C\cap \widetilde{C}$ and $C, \widetilde{C} \not\subset \BP^{n}.$ Let $A \subset X$ (resp. $\widetilde{A} \subset \widetilde{X}$) be the proper transform of $C$ (resp. $\widetilde{C}$). Here when $S$ and $\widetilde{S}$ are not closed submanifolds of $\BP^{N}$, we consider the blowup $X$ (resp. $\widetilde{X}$) as defined only over  a neighborhood of $C$ (resp. $\widetilde{C}$)  in $\BP^{n+1}$.  Then $(A/X)_k $ is biholomorphic to $(\widetilde{A}/\widetilde{X})_k$. \end{proposition}

 \begin{proof}
   Note that the subgroup  of ${\rm PGL}(n+2)$ consisting of the projective transformations of $\BP^{n+1}$ fixing the hyperplane $\BP^{n}$ pointwise acts transitively on the affine cell $U:= \BP^{n+1} \setminus \BP^{n}$ and this action  can be lifted to an action on $X$ (resp. $\widetilde{X}$).  As some element of this group sends $\widetilde{C}$ to $C$, we may assume that $C = \widetilde{C}$.

By the assumptions on $S$ and $\widetilde{S}$, we can choose an open  neighborhood $O$ (resp. $\widetilde{O}$) of $x \in \BP^{n+1}$ with holomorphic coordinates $(t_1, \ldots, t_{n+1})$ (resp. $\widetilde{t}_1, \ldots, \widetilde{t}_{n+1}$) centered at $x$ such that \begin{itemize}
\item[(1)]  $t_{n+1} = \widetilde{t}_{n+1}$ on $O\cap \widetilde{O}$ and  the hyperplane $\BP^n \cap O \cap \widetilde{O}$ is defined by
$t_{n+1} =0 = \widetilde{t}_{n+1}$; \item[(2)]
$t_{1} = \widetilde{t}_{1}, t_{2} = \widetilde{t}_{2}, \ldots, t_m = \widetilde{t}_m$ on $O \cap \widetilde{O}$;
\item[(3)]
$S \cap O$ (resp. $\widetilde{S} \cap \widetilde{O}$) is defined by   $$ t_{m+1} = t_{m+2} = \cdots = t_n = t_{n+1} =0 $$ $$\mbox{ ( resp. } \widetilde{t}_{m+1} = \widetilde{t}_{m+2} = \cdots = \widetilde{t}_n = \widetilde{t}_{n+1} =0 \mbox{ )};$$
\item[(4)] $C \cap O$ (resp. $C \cap \widetilde{O}$) is defined by $$t_1= \cdots = t_n =0 \mbox{ ( resp. } \widetilde{t}_1 = \cdots = \widetilde{t}_n=0 \mbox{ )};$$
\item[(5)] $t_{m+1} = \widetilde{t}_{m+1}, t_{m+2}   = \widetilde{t}_{m+2}, \cdots, t_n = \widetilde{t}_n,$ when they are restricted to $(C\cap O \cap \widetilde{O}/ O \cap \widetilde{O})_k$.
   \end{itemize}
   Fix inhomogeneous coordinates $(u_1, \ldots, u_{n+1})$ on $U$ such that $C \cap U$ is defined by $u_1 = \cdots = u_{n} =0$. On $U \cap O$ (resp. $U \cap \widetilde{O}$), we have transition functions of coordinates
   $$u_i = G_i(t_1, \ldots, t_{n+1}) \mbox{ ( resp. }  u_i = \widetilde{G}_i(\widetilde{t}_1, \ldots, \widetilde{t}_{n+1}) \mbox{ )}$$ for $1 \leq i \leq n+1$.
   By our choices of coordinates, the functions $G_i$ and $\widetilde{G}_i$ coincide on
   $(C \cap U \cap O \cap \widetilde{O}/U\cap O \cap \widetilde{O})_k$.

Let $E \subset X$ (resp. $\widetilde{E} \subset \widetilde{X}$) be the exceptional divisor of the blowup $X \to \BP^{n+1}$ along $S$ (resp. $\widetilde{X} \to \BP^{n+1}$ along $\widetilde{S}$).
Identify $U$ with an open subset  ${\bf U} \subset X$ (resp. $\widetilde{\bf U} \subset \widetilde{X}$) in a natural way with coordinates ${\bf u}_1, \ldots, {\bf u}_{n+1}$ (resp. $\widetilde{\bf u}_1, \ldots, \widetilde{\bf u}_{n+1}$) induced by $u_1, \ldots, u_{n+1}$.
The coordinates $t_1, \ldots, t_{n+1}$ and $\widetilde{t}_1, \ldots, \widetilde{t}_{n+1}$ induce coordinates ${\bf t}_1, \ldots, {\bf t}_{n+1}$ in a neighborhood ${\bf O}$  of $A \cap E$ in $X$  and coordinates $\widetilde{\bf t}_1, \ldots, \widetilde{\bf t}_{n+1}$ in a neighborhood  $\widetilde{\bf O}$ of $\widetilde{A} \cap \widetilde{E}$ in $\widetilde{X}$ via quadratic transformations of the blowup:
$$ {\bf t}_1 = t_1, \ldots, {\bf t}_m = t_m, {\bf t}_{m+1} = \frac{t_{m+1}}{t_{n+1}}, \ldots, {\bf t}_n= \frac{t_n}{t_{n+1}}, {\bf t}_{n+1} = t_{n+1}, $$ $$ \widetilde{\bf t}_1 = \widetilde{t}_1, \ldots, \widetilde{\bf t}_m = \widetilde{t}_m, \widetilde{\bf t}_{m+1} = \frac{\widetilde{t}_{m+1}}{\widetilde{t}_{n+1}}, \ldots, \widetilde{\bf t}_n= \frac{\widetilde{t}_n}{\widetilde{t}_{n+1}}, \widetilde{\bf t}_{n+1} = \widetilde{t}_{n+1}.  $$
 We have $A \subset {\bf U} \cup {\bf O}$  and $\widetilde{A} \subset \widetilde{\bf U} \cup \widetilde{\bf O}$. The coordinate transition functions $${\bf u}_i = {\bf G}_i({\bf t}_1, \ldots, {\bf t}_{n+1}) \ \mbox{ and }
\widetilde{\bf u}_i = \widetilde{\bf G}_i(\widetilde{\bf t}_1, \ldots, \widetilde{\bf t}_{n+1}) $$
are obtained from  $G_i$ and $\widetilde{G}_i$ by the substitutions   $$t_1 = {\bf t}_1,  \ldots, t_m= {\bf t}_m, t_{m+1} = {\bf t}_{n+1} {\bf t}_{m+1}, \ldots, t_n = {\bf t}_{n+1} {\bf t}_n, t_{n+1} = {\bf t}_{n+1},$$ $$ \widetilde{t}_1 = \widetilde{\bf t}_1, \ldots, \widetilde{t}_m = \widetilde{\bf t}_m, \widetilde{t}_{m+1}=  \widetilde{\bf t}_{n+1} \widetilde{\bf t}_{m+1}, \ldots, \widetilde{t}_n = \widetilde{\bf t}_{n+1} \widetilde{\bf t}_n, \widetilde{t}_{n+1} = \widetilde{\bf t}_{n+1}.  $$
Since $G_i$ and $\widetilde{G}_i$ coincide on
   $(C \cap U \cap O \cap \widetilde{O}/U\cap O \cap \widetilde{O})_k,$ the functions ${\bf G}_i$ and $\widetilde{\bf G}_i$ coincide on $$(A \cap {\bf U} \cap {\bf O} / {\bf U} \cap {\bf O})_k  \cap
      (\widetilde{A} \cap \widetilde{\bf U} \cap \widetilde{\bf O}/\widetilde{\bf U} \cap \widetilde{\bf O})_k$$ after the natural identification $U = {\bf U} = \widetilde{\bf U}$.
Thus they define biholomorphic structures on $(A/X)_k$ and $(\widetilde{A}/\widetilde{X})_k$
(see Remark 1 in page 302 of \cite{MR}).
\end{proof}

Using Proposition \ref{p.example}, we can relate Definition \ref{d.contact} (iii) to Definition \ref{d.iso} (2) as follows.

 \begin{theorem}\label{t.example}
Let $S, \widetilde{S} $ be two submanifolds of $\BP^{n}$ which have contact up to order $k > 0$.
Regarding $\BP^{n}$ as a hyperplane in $\BP^{n+1}$, let $X$ (resp. $\widetilde{X}$) be the blowup of $\BP^{n+1}$ along $S$ (resp. $\widetilde{S}$).
Let $\sK$ (resp. $\widetilde{\sK}$) be the connected open subset of ${\rm Douady}(X)$ (resp.
${\rm Douady}(\widetilde{X})$) parametrizing proper transforms of lines in $\BP^{n+1}$ intersecting $S$ (resp. $\widetilde{S}$) which are not contained in $\BP^{n}$. Here, as in Theorem \ref{t.example}, if $S$ and $\widetilde{S}$ are not closed submanifolds of $\BP^{n}$, we consider the blowup $X$ (resp. $\widetilde{X}$) as defined only in  a neighborhood of some member of $\sK$ (resp. $\widetilde{\sK}$).  Then $\sK$ and $\widetilde{\sK}$ are iso-equivalent up to order $k$. \end{theorem}

\begin{proof}
Fix a hyperplane ${\bf P}^{n} \subset \BP^{n+1}$ different from $\BP^n \subset \BP^{n+1}$.
To prove the theorem, we may replace $S$ and $\widetilde{S}$ by their open subsets such  that $$S_o= S, \ \widetilde{S}_o=\widetilde{S}, \ S \cap {\bf P}^n =  \emptyset = \widetilde{S} \cap {\bf P}^n$$ in the notation of Definition \ref{d.contact} (iii).   Thus we have a biholomorphic map $h: S \to \widetilde{S}$ and a holomorphic map $g: S \to {\rm PGL}(n+1)$ such that $S$ and $g(s) \cdot \widetilde{S}$ agree at $s = g(s)\cdot h(s)$ up to order $k$ for each $s \in S$.
Furthermore, we can fix an inclusion ${\rm PGL}(n+1) \subset {\rm PGL}(n+2)$ as a closed subgroup and regard each $g(s)$ as a projective transformation of $\BP^{n+1}$ preserving $\BP^n$.

Fix a connected open subset $U \subset {\bf P}^n \setminus \BP^n$. For each $s \in S$ (resp. $\widetilde{s} \in \widetilde{S}$) and $u \in U$, denote by $A(u,s) \subset X$ (resp. $\widetilde{A}(u, \widetilde{s}) \subset \widetilde{X}$) the proper transformation in $X$ (resp. $\widetilde{X}$) of the line joining $u$ and $s$ (resp. $\widetilde{s}$). Define
connected open subsets $W \subset \sK$ and $\widetilde{W} \subset \widetilde{\sK}$ by
$$ W := \{  [A(u,s)] \in \sK, \ u \in U, s \in S\} \cong U \times S,$$ $$ \widetilde{W}:= \{ [\widetilde{A}(u, \widetilde{s})] \in \widetilde{\sK}, \ u \in U, \widetilde{s} \in \widetilde{S}\} \cong U \times \widetilde{S}.$$ Define the biholomorphic map
$f: W \to \widetilde{W}$ by $f([A(u,s)]) = [\widetilde{A}(u, h(s))]$ for $[u,s] \in W$.

Let $X'$ be the blowup of $\BP^{n+1}$ along $S' := g(s) \cdot \widetilde{S}$ and let $A' \subset X'$ be the proper transformation of the line joining $g(s) \cdot u$ to $s$.
Then we have biholomorphic maps
$$ (A(u,s)/X)_k \cong (A'/X')_k \cong (\widetilde{A}( u, h(s))/\widetilde{X})_k $$ for each $[A(u,s)]  \in W$ where the first biholomorphic map is from  Proposition \ref{p.example} and the second biholomorphic map is induced by the projective transformation $g(s) \in {\rm PGL}(n+2)$.  This defines the biholomorphic map $F_k: \rho_k^{-1}(W) \cong \widetilde{\rho}_k^{-1}(\widetilde{W})$, proving that $\sK$ and $\widetilde{\sK}$ are iso-equivalent up to order $k$.
\end{proof}

Obviously, any two submanifolds $S, \widetilde{S} \subset \BP^n$ of equal dimension have contact up to order 1. Thus the families $\sK$ and $\widetilde{\sK}$ in Theorem \ref{t.example} are iso-equivalent up to order 1. This is a special case of Proposition \ref{p.mrc} because
 members of $\sK, \widetilde{\sK}$ in Theorem  \ref{t.example} are unbendable rational curves (see Lemma \ref{l.Hw10} below).

We introduce the following generalization of the notion of varieties of minimal rational tangents  in \cite{HM98}.

\begin{definition}\label{d.vmrt} Let $A \subset X$ be an unbendable nonsingular rational curve.
Let $[T_{A,y}] \in \BP T_{X,y}$ be the point corresponding to the tangent space of $A$ at $y \in A$.
For any point $y \in A$, the set of tangent spaces to deformations of $A$ fixing $y$ in $X$
form a germ of $m$-dimensional submanifold $\sC^A_y \subset \BP T_{X,y}$ containing the point $[T_{A,y}]$, called the {\em variety of minimal rational tangents} (abbr. VMRT) of $A$ at $y$.
It is clear that if $\Phi: (A/X)_{\sO} \to (\widetilde{A}/\widetilde{X})_{\sO}$ is a biholomorphic map,
then for each $y \in A$,  the germs of submanifolds $$[T_{A,y}] \in \sC_y^A \subset \BP T_{X,y} \mbox{ and }
 [T_{\widetilde{A}, \Phi(y)}] \in \sC_{\Phi(y)}^{\widetilde{A}} \subset \BP T_{\widetilde{X}, \Phi(y)}$$ are projectively isomorphic via  ${\rm d}_y \Phi: \BP T_{X,y} \cong \BP T_{\widetilde{X}, \Phi(y)}.$
\end{definition}

A version of the following lemma has appeared in Example 1.7 of \cite{Hw10}.

\begin{lemma}\label{l.Hw10}
Let $S, X, A$ be as in Proposition \ref{p.example}. Then $A$ is unbendable and  for any point $y \in A$ not on the exceptional divisor of the blowup $X \to \BP^{n+1}$, the VMRT $[T_{A,y}] \in \sC^A_y \subset \BP T_{X,y} \cong \BP^{n}$    is projectively isomorphic to $x \in S \subset  \BP^{n}$ as germs of submanifolds in projective space. In particular, if the germs $x \in S\subset \BP^{n}$ and $x \in \widetilde{S} \subset \BP^{n}$ are not projectively isomorphic, then $(A/X)_{\sO}$ is not biholomorphic to $(\widetilde{A}/\widetilde{X})_{\sO}$. \end{lemma}

\begin{proof}
As $C$ intersects $S$ transversally and $N_{C/\BP^{n+1}} \cong \sO(1)^{\oplus n}$, the normal bundle of the proper transform $A$ of $C$ under the blowup $\beta: X \to \BP^{n+1}$ along the $m$-dimensional submanifold $S \subset \BP^{n+1}$ has the isomorphism type
$$N_{A/X} \cong \sO(1)^{\oplus m} \oplus \sO^{\oplus (n-m)}.$$
Thus $A$ is unbendable and $\dim \sC^A_y = m$. Consider  the collection $\sL$ of all lines in $\BP^{n+1}$ connecting the point $z=\beta(y)$ to $S$. The set $$\sC^{\sL} \subset \BP T_{\BP^{n+1},z}$$ of tangent spaces at $z$ of members of $\sL$ is isomorphic to $S \subset \BP^n$ as projective submanifolds
such that the germ of $[T_{C,z}] \in \sC^{\sL} \subset \BP T_{ \BP^{n+1},z}$ is isomorphic to that of $x \in S \subset \BP^n$. The proper transforms of members of $\sL$ are deformations of $A$ fixing $y$ and they give all the local deformations of $A$ fixing $y$ because  $\dim \sL = m$.  It follows that $[T_{A,y}] \in \sC^A_y \subset \BP T_{X,y} \cong \BP^{n}$    is projectively isomorphic to the germ
$[T_{C,z}] \in \sC^{\sL} \subset \BP T_{ \BP^{n+1},z}$, hence to the germ
$x \in S \subset  \BP^{n}$. \end{proof}

Unbendable rational curves are free and their normal bundles are separating in the sense of Theorem \ref{t.1}, if the VMRT has positive dimension.
Thus  Lemma \ref{l.Hw10} enables us to  apply Theorem \ref{t.1} in the setting of Theorem \ref{t.example}
(see Remark \ref{r.JeVe}). It is also useful in the following examples.

\begin{example}\label{e.blowup}
Let $k \geq 2$ be a positive integer.
Let $S \subset \BP^2$ be a line and let $\widetilde{S} \subset \BP^{2}$ be an irreducible algebraic curve of degree $>1$ such that $S$ and $\widetilde{S}$ agree up to order $k$ at a point $x \in S \cap \widetilde{S}.$ As in Proposition \ref{p.example},
view $\BP^2$ as a hyperplane in $\BP^3$ and fix a line $C \subset \BP^3$ such that $x \in C$ and $ C \not\subset \BP^2$. Let $X$ (resp. $\widetilde{X}$) be the blowup of $\BP^3$ along $S$ (resp. $\widetilde{S}$) and let $A \subset X$ (resp. $\widetilde{A} \subset \widetilde{X}$) be the proper image of $C$ under the blowup. Then $(A/X)_k \cong (\widetilde{A}/\widetilde{X})_k$ by Proposition \ref{p.example}, but $(A/X)_{\sO} \not\cong (\widetilde{A}/\widetilde{X})_{\sO}$ by Lemma \ref{l.Hw10}. This shows that  Question \ref{q.free} has negative answer for the  free rational curve $A \subset X$.  \end{example}

\begin{example}\label{e.t.1}
Let $C_k \subset \BP^k$ be the rational normal curve of degree $k$, i.e., the closure of the image of the holomorphic map in affine coordinates $$ \C \ni t  \mapsto (x_1= t, x_2= t^2, \ldots, x_k = t^k).$$ Regarding $\BP^k $ as a hyperplane in $\BP^{k+1}$, we see that $C_k$ and $C_{k+1} \subset \BP^{k+1}$ have contact up to order $k$ as submanifolds of $\BP^{k+1}$. Then by Theorem \ref{t.example}, we obtain a free family $\sK$ on the blowup of $\BP^{k+2}$ along $C_k$ and another free family $\widetilde{\sK}$ on the blowup of $\BP^{k+2}$ along $C_{k+1}$, which are  iso-equivalent up to order $k$. As $\sK$ and $\widetilde{\sK}$ are not germ-equivalent by Lemma \ref{l.Hw10}, we see that the integer $\ell(\sK)$ in Theorem \ref{t.1} for this $\sK$  is at least $k+1$. \end{example}

\section{Equivalence problem for holomorphic G-structures}\label{s.Gstr}

In this section, we recall basic notions in the theory of holomorphic G-structures and present a key result, Theorem \ref{t.G}, which is a consequence of Morimoto's work \cite{Mor} on the equivalence problem of geometric structures.

\begin{definition}\label{d.Frame}
Let $V$ be a fixed vector space of dimension $n$.
Let $M$ be an $n$-dimensional complex manifold and let $\sR(M)$ be its frame bundle, i.e., a principal ${\rm GL}(V)$-bundle with the fiber at $x \in M$ is defined by
$$\sR_x(M) = {\rm Isom}(V, T_{M,x}).$$
A biholomorphic map $\Phi: M \to \widetilde{M}$ between complex manifolds induces a biholomorphic map
$\Phi_*: \sR(M) \to \sR(\widetilde{M})$
which sends an element $f \in {\rm Isom}(V, T_x(M))$ to the composition $$[ \Phi_*(f) : V \stackrel{f}{\to} T_{M,x} \stackrel{{\rm d}_x \Phi}{\to} T_{\widetilde{M}, \Phi(x)}] \ \in \sR_{\Phi(x)}(\widetilde{M}).$$
\end{definition}

\begin{lemma}\label{l.coordi}
Let $x \in M$ and $\widetilde{x} \in \widetilde{M}$ be points on complex manifolds. Fix a positive integer $\ell$. Then an isomorphism of complex spaces $\varphi:(x/M)_{\ell} \to (\widetilde{x}/\widetilde{M})_{\ell}$ induces a natural isomorphism of the infinitesimal neighborhoods of the fibers $$ \varphi_*: (\sR_x(M)/\sR(M))_{\ell-1} \to (\sR_{\widetilde{x}}(\widetilde{M})/\sR(\widetilde{M}))_{\ell-1}$$ such that
if $\varphi$ comes from a biholomorphic map $\Phi: (x/M)_{\sO} \cong (\widetilde{x}/\widetilde{M})_{\sO}$, then
$\varphi_*$ is just the restriction of $\Phi_*$  in Definition \ref{d.Frame}. \end{lemma}

\begin{proof}
We can extend $\varphi$ to a biholomorphic map $\Phi: (x/M)_{\sO} \to (\widetilde{x}/\widetilde{M})_{\sO}$ which induces $$\Phi_*: (\sR_x(M)/\sR(M))_{\sO} \to (\sR_{\widetilde{x}}(\widetilde{M})/\sR(\widetilde{M}))_{\sO}.$$ Define
$\varphi_*$ as the restriction of $\Phi_*$ to $(\sR_x(M)/\sR(M))_{\ell -1}$. It remains to show that this definition of $\varphi_*$ does not depend on the choice of the extension $\Phi$.
Let us check it by local coordinates.

Let $n$ be $\dim V= \dim M= \dim \widetilde{M}.$  Fix a basis of $V$ once and for all to identify a point of $\sR_y(M), y \in M$ (resp. $\sR_{\widetilde{y}}(\widetilde{M}), \widetilde{y} \in \widetilde{M}$) with a basis of $T_{M,y}$ (resp. $T_{\widetilde{M}, \widetilde{y}}$). Choose coordinates $(x^1, \ldots, x^n)$ in a neighborhood of  $x \in M$ and  coordinates  $(\widetilde{x}^1, \ldots, \widetilde{x}^n)$ in a neighborhood of $\widetilde{x} \in \widetilde{M}$. They give matrices of local holomorphic functions $$[p^i_j, 1 \leq i, j \leq n] \mbox{ on }(\sR_x(M)/\sR(M))_{\sO}$$  $$[\widetilde{p}^i_j, 1 \leq i, j \leq n]\mbox{ on }(\sR_{\widetilde{x}}(\widetilde{M})/ \sR(\widetilde{M}))_{\sO}$$ such that
  $$\{\sum_{i=1}^n p^i_j \frac{\partial}{\partial x^i}, \ 1 \leq j \leq n\} \mbox{ and } \{\sum_{i=1}^n \widetilde{p}^i_j \frac{\partial}{\partial \widetilde{x}^i}, \ 1 \leq j \leq n\}$$ describe bases of $T_{M,y}$ and  $T_{\widetilde{M}, \widetilde{y}}$ for $y$ in a neighborhood of $x\in M$ and $\widetilde{y} = \Phi(y)$.  Thus $( p^i_j, x^{k}, \ 1 \leq i,j,k \leq n)$ can be viewed as local coordinates on $(\sR_x(M)/\sR(M))_{\sO}$  and
$(\widetilde{p}^i_j, \widetilde{x}^{k}, \ 1 \leq i,j,k \leq n)$ can be viewed as local coordinates on  $(\sR_{\widetilde{x}}(\widetilde{M})/\sR(\widetilde{M}))_{\sO}$.  If the map $\Phi$ is given by
$\widetilde{x}^{i} = \Phi^{i}(x^1, \ldots, x^n)$ for some holomorphic functions $\Phi^1, \ldots, \Phi^n$, then the map $\Phi_*: \sR(M) \to \sR(\widetilde{M})$ is  given by $$ \widetilde{p}^i_j = \sum_{k =1}^n p_j^{k} \frac{\partial \Phi^i}{\partial x^{k}}, \ \widetilde{x}^{i} = \Phi^{i} (x^1, \ldots, x^n), $$ in these local coordinates. Thus $\Phi_*|_{(\sR_x(M)/\sR(M))_{\ell -1}}$ is determined by $\Phi|_{(x/M)_{\ell}}$.
\end{proof}

\begin{definition}\label{d.G}
Let $G \subset {\rm GL}(V)$ be a closed  subgroup. A $G$-structure on a complex manifold $M$ with $\dim M = \dim V$ means a principal $G$-subbundle $\sP \subset \sR(M).$
Let $\sP \subset \sR(M)$ and $\widetilde{\sP} \subset \sR(\widetilde{M})$ be $G$-structures on two complex manifolds $M$ and $\widetilde{M}$. \begin{itemize} \item[(1)] We say that  $\sP$ and $\widetilde{\sP}$ are {\em locally equivalent} (or $\sP$ is {\em locally equivalent} to $\widetilde{\sP}$), if there exist \begin{itemize} \item[(1a)] nonempty connected open subsets $O \subset M$ and $\widetilde{O} \subset \widetilde{M}$; and \item[(1b)] a biholomorphic map $\Phi: O\to \widetilde{O}$ such that the induced biholomorphic map $\Phi_*: \sR(O) \to \sR(\widetilde{O})$ sends $\sP|_O$ to $\widetilde{\sP}|_{\widetilde{O}}.$ \end{itemize} \item[(2)] We say that $\sP$ and $\widetilde{\sP}$ are {\em locally equivalent up to order} $k $ (or $\sP$ is {\em locally equivalent  to $\widetilde{\sP}$ up to order} $k$), if there exist \begin{itemize} \item[(2a)] nonempty connected open subsets $U \subset M$ and $\widetilde{U} \subset \widetilde{M}$; \item[(2b)] a biholomorphic map $\Psi: U \to \widetilde{U}$; and \item[(2c)]  a holomorphic family of biholomorphic maps $$ \{ \psi^x: (x/M)_k \to (\widetilde{x}/ \widetilde{M})_k, \ x \in U, \widetilde{x} = \Psi(x)\}$$  such that $\psi^x_*$ in the notation of Lemma \ref{l.coordi} sends
 $$\sP \cap (\sR_x(M)/\sR(M))_{k-1} \mbox{ to } \widetilde{\sP} \cap (\sR_{\widetilde{x}}(\widetilde{M})/ \sR(\widetilde{M}))_{k-1},$$ or equivalently, $ \psi^x_{*} ( \sP_x/\sP)_{k-1}  = (\widetilde{\sP}_{\widetilde{x}}/\widetilde{\sP})_{k-1}.$ \end{itemize} \end{itemize}\end{definition}

It is easy to see that any two $G$-structures $\sP$ and $\widetilde{\sP}$ are locally equivalent up to order 1.
It is clear that if $\sP$ and $\widetilde{\sP}$ are locally equivalent, then they are locally equivalent up to order $k$ for any positive integer $k$. In fact, given $\Phi$ as in Definition \ref{d.G} (1), we can just choose $$\Psi = \Phi \mbox{ on } U=O, \ \psi^x = \Phi_*|_{(\sP_x/\sP)_{k-1}}.$$ The following theorem says the converse holds if $k$ is sufficiently large.

\begin{theorem}\label{t.G} For a vector space $V$ of dimension $n$, let  $G \subset {\rm GL}(V)$ be a closed subgroup.
Let $M$ be an $n$-dimensional complex manifold with a holomorphic $G$-structure $\sP$.
Then there exist  a positive integer $k_o$  such that
if $\sP$ and  a $G$-structure $\widetilde{\sP}$  on a complex manifold $\widetilde{M}$ are locally equivalent up to order $k_o$, then they are locally equivalent. \end{theorem}

This theorem is essentially contained in Morimoto's paper \cite{Mor}. As Morimoto has not stated it explicitly, we give a sketch of  the proof, with precise references  to \cite{Mor}. The reader may find it useful to look into Section 2 of \cite{Hw19} for a streamlined review of the relevant part of Morimoto's paper.

\begin{proof}[Proof of Theorem \ref{t.G}]
Before starting the proof, let us mention that we are given open subsets $U, \widetilde{U}$ of Definition \ref{d.G} (2) and need to find  open subsets $O, \widetilde{O}$ of Definition \ref{d.G} (1). In the proof below, the open subsets $O, \widetilde{O}$ are chosen to be strictly smaller than $U, \widetilde{U}$ to avoid certain nowhere-dense subsets in $U$ where the geometric structure cannot be lifted to an involutive structure.

 Suppose the $G$-structure $\sP$ is involutive in the sense that it satisfies the conditions in Theorem 8.2 of \cite{Mor}. We claim that $k_o =3$ works:
  any $G$-structure $\widetilde{\sP}$ on a manifold $\widetilde{M}$ which is locally equivalent to $\sP$ up to order 3 must be locally equivalent to $\sP$.
   To prove the claim, we may shrink $M$ and $\widetilde{M}$ to assume that we have biholomorphic maps $\Psi: M \cong \widetilde{M}$ and $\{\psi^x, x \in M\}$ satisfying Definition \ref{d.G} (2) with $k=3$.  Recall that the involutiveness of a $G$-structure (Theorem 8.2 of \cite{Mor}) is determined by the group $G \subset {\rm GL}(V)$ and the structure function of the $G$-structure. The values of the structure function  of $\sP \subset \sR(M)$ on the fiber $\sP_x, x \in M,$ is determined by $\sP \cap (\sR_x(M)/\sR(M))_2$, because it is calculated from the derivative of a natural 1-form (the fundamental form, i.e., the soldering form) on $\sR(M).$   Thus we see that  $\widetilde{\sP}$ is also involutive and have the same values of the structure function as those of $\sP$.   Thus $\sP$ and $\widetilde{\sP}$ are locally equivalent by Theorem 8.2 of \cite{Mor} (attributed to Singer-Sternberg and Tanaka-Ueno).

In the general case when $\sP$ is not  involutive, Theorem 9.1 of \cite{Mor} gives an algorithm consisting of a finite number of operations involving prolongations and reductions, through which we reach an involutive Cartan bundle of higher-order (in the sense of Definition 8.1 of \cite{Mor}) outside a nowhere-dense subset $S \subset M$. Each step of this algorithm is determined by a finite-order neighborhood of each point of $M$ coming from Lemma \ref{l.coordi}, the order of which can be explicitly decided. Thus there exists a positive integer $k_o$ determined by $\sP$ such that the structure of the resulting involutive Cartan bundle at each point $x \in M \setminus S$ is determined by the $k_o$-th order neighborhood  $(x/\sM)_{k_o}$. Now let $\widetilde{\sP}$ be a $G$-structure on $\widetilde{M}$ which is locally equivalent to $\sP$ up to  order $k_o$ via $\Psi: U \to \widetilde{U}$ in the sense of Definition \ref{d.G} (2). If we apply the algorithm of Theorem 9.1 of \cite{Mor} to $\widetilde{\sP}$ in the same manner as $\sP$,  we reach an involutive Cartan bundle at exactly the same number of steps as $\sP$ on $\Psi(U \setminus S)$ and the structures of the involutive Cartan bundles   arising from $\sP$ and $\widetilde{\sP}$ are isomorphic on $U \setminus S$ and $\Psi(U \setminus S)$ from Theorem 8.1 of
\cite{Mor}.  By the uniqueness part of Theorem 9.1 of \cite{Mor}, we can  find nonempty  open subsets $O \subset (U \setminus S)$ and $\widetilde{O} = \Psi(O)$ with a biholomorphic map $\Phi: O \to \widetilde{O}$ satisfying the condition of Definition \ref{d.G} (1).  \end{proof}

\section{Proof of Theorems \ref{t.1} and \ref{t.CF}}

We can reformulate Theorem \ref{t.1} in terms of Definition \ref{d.iso} as follows.

 \begin{theorem}\label{t.family}
 Let $\sK \subset {\rm Douady}(X)$ be a  free family  in a complex manifold $X$ whose members have separating normal bundles. Then there exists a positive integer $\ell=\ell(\sK)$ such that for any  free family $\widetilde{\sK} \subset {\rm Douady}(\widetilde{X})$  in a  complex manifold $\widetilde{X}$, if $\sK$ and $\widetilde{\sK}$ are  iso-equivalent up to order $\ell$, then they are germ-equivalent. \end{theorem}

The proof of Theorem \ref{t.family} is by applying Theorem \ref{t.G} to the G-structure on the universal family $\sU$ of Definition \ref{d.iso} determined by the fibrations $\rho$ and $\mu$ in the following way.

\begin{definition}\label{d.sP} In Definition \ref{d.iso}, consider the two vector subbundles of $T_{\sU},$  $$T^{\rho} := {\rm Ker} ({\rm d} \rho) \mbox{ and } T^{\mu} := {\rm Ker}({\rm d} \mu).$$ They satisfy $T^{\rho} \cap T^{\mu} =0$. Fix a vector space $V$ with two subspaces $V^{\rho}, V^{\mu} \subset V$ such that $\dim V = \dim \sU, \ \dim V^{\rho} = {\rm rank } (T^{\rho}), \ \dim V^{\mu} = {\rm rank } (T^{\mu})$ and $  V^{\rho} \cap V^{\mu} =0.$
Let  $\sP^{\rho}$ (resp. $\sP^{\mu}$) be the fiber subbundle of the frame bundle $\sR(\sU)$ whose fiber
 $\sP^{\rho}_y$ (resp. $\sP^{\mu}_y$) at $y \in \sU$ is
   $$ \sP^{\rho}_y := \{h \in {\rm Isom}(V, T_{\sU,y}) = \sR_y(\sU), \ h(V^{\rho}) = T^{\rho}_y \} $$ $$ \mbox{ ( resp. } \sP^{\mu}_y : = \{ h \in {\rm Isom}(V, T_{\sU, y}) = \sR_y(\sU), \ h(V^{\mu}) = T^{\mu}_y \} \mbox{).} $$
   Then define $\sP^{\rho,\mu} := \sP^{\rho} \cap \sP^{\mu}$, a  $G$-structure on $\sU$ whose structure group $G$ is the  subgroup of ${\rm GL}(V)$ preserving $V^{\rho}$ and $V^{\mu}$.
\end{definition}

To prove Theorem \ref{t.family}, we use the following lemma describing a key geometric property of the $G$-structure in Definition \ref{d.sP} arising from the separating condition on the normal bundles of members of the family $\sK$. This lemma was proved in Proposition 3.3 of \cite{Hw19} using an argument generalizing the proof of the Cartan-Fubini type extension theorem in \cite{HM01}.

    \begin{lemma}\label{l.Hw}
    Let $\sK \stackrel{\rho}{\leftarrow} \sU \stackrel{\mu}{\rightarrow} X$ and
    $\widetilde{\sK} \stackrel{\widetilde{\rho}}{\leftarrow} \widetilde{\sU} \stackrel{\widetilde{\mu}}{\rightarrow} \widetilde{X}$ be two  free families, whose members have separating normal bundles. Let $\sP = \sP^{\rho, \mu}$ on $\sU$  and $\widetilde{\sP} = \sP^{\widetilde{\rho}, \widetilde{\mu}}$ on $\widetilde{\sU}$ be the G-structures  defined in Definition \ref{d.sP}.  Assume that $\sP$ and $\widetilde{\sP}$ are locally equivalent by a biholomorphic map  $\Phi: O \to \widetilde{O}$ as in Definition \ref{d.G} (1)
      for some connected nonempty open subsets $O \subset \sU$ and $\widetilde{O} \subset \widetilde{\sU}$. Then after shrinking $O$ and $\widetilde{O}$ if necessary,   \begin{itemize}
    \item[(i)]
    $\Phi$ sends fibers of $\mu|_O$ to fibers of $\widetilde{\mu}|_{\widetilde{O}}$ descending to a biholomorphic map $\Phi^{\flat}: \mu(O) \to \widetilde{\mu}(\widetilde{O})$;
        \item[(ii)] $\Phi$ sends fibers of $\rho|_O$ to fibers of $\widetilde{\rho}|_{\widetilde{O}}$
        descending to a biholomorphic map $\Phi^{\sharp}: \rho(O) \to \widetilde{\rho}(\widetilde{O})$;
        \item[(iii)] for each $z \in O$ and $\widetilde{z} = \Phi(z)$, there exists a biholomorphic map $F: W \to \widetilde{W}$ from a neighborhood $W$ of $\rho^{-1}(\rho(z))$ to a neighborhood $\widetilde{W}$ of $\widetilde{\rho}^{-1}(\widetilde{\rho}(\widetilde{z}))$ such that the germ of $F$ at $z$ equals the germ of $\Phi$ at $z$; and \item[(iv)]
            setting $A := \mu(\rho^{-1}(\rho(z)) \subset X$ and $\widetilde{A}:= \widetilde{\mu}(\widetilde{\rho}^{-1}(\widetilde{\rho}(\widetilde{z}))) \subset \widetilde{X},$
            the map $F$ induces a biholomorphic map $F^{\flat}: (A/X)_{\sO} \cong (\widetilde{A}/\widetilde{X})_{\sO}$ which agrees with $\Phi^{\flat}$ on $(\mu(z)/X)_{\sO}$.  \end{itemize} \end{lemma}

We skip the proof of the following elementary fact, which can be seen from the proof of the holomorphic implicit function theorem (e.g. Theorem I.1.18 in \cite{GLS}).

\begin{lemma}\label{l.GLS}
Let $\alpha: M \to B$ and $\widetilde{\alpha}: \widetilde{M} \to \widetilde{B}$ be two smooth surjective morphisms of complex manifolds. Fix a positive integer $k$. For $z \in M, b = \alpha(z), \widetilde{z} \in \widetilde{M}, \widetilde{b} = \widetilde{\alpha}(\widetilde{z})$, suppose that we have biholomorphic maps
$$ \gamma: (z/M)_k \to (\widetilde{z}/\widetilde{M})_k \mbox{ and }
\psi: (b/B)_k \to (\widetilde{b}/\widetilde{B})_k $$ satisfying $\psi \circ \alpha|_{(z/M)_k} = \widetilde{\alpha} \circ \gamma$. Then we can find  biholomorphic maps
$$ \Gamma: (z/M)_{\sO} \to (\widetilde{z}/\widetilde{M})_{\sO} \mbox{ and }
\Psi: (b/B)_{\sO} \to (\widetilde{b}/\widetilde{B})_{\sO} $$ satisfying $\Psi \circ \alpha|_{(z/M)_{\sO}} = \widetilde{\alpha} \circ \Gamma, \Gamma|_{(z/M)_k} = \gamma$ and $\Psi|_{(b/B)_k} = \psi$. \end{lemma}

\begin{proof}[Proof of Theorem \ref{t.family}]
We have the G-structures $\sP$ on $\sU$ and $\widetilde{\sP}$ on $\widetilde{\sU}$ from Lemma \ref{l.Hw}.
  Put $\ell(\sK) = k_o$ where the integer $k_o$ is from Theorem \ref{t.G} applied to  the G-structure $\sP$ on $M= \sU$. Assuming that $\sK$ and $\widetilde{\sK}$ are  iso-equivalent up to order $k \geq k_o$, we prove that they are germ-equivalent as follows.

In the terminology of Definition \ref{d.iso} (2),
for each member $[A]$ of $\sW \subset \sK$, we have the corresponding member $[\widetilde{A}]\in \widetilde{\sW} \subset \widetilde{\sK}$
 with an isomorphism $\varphi: (A/X)_k \to (\widetilde{A}/\widetilde{X})_k$.
 We have the induced isomorphism $\varphi^{\sU}$ and $\varphi^{\sK}$ from Lemma \ref{l.Hir} which commutes with $\rho, \widetilde{\rho}, \mu$ and $\widetilde{\mu}$.  Its restriction $$\varphi_u : (u/\sU)_k \to (\widetilde{u}/\widetilde{\sU})_k, \ \widetilde{u} := \varphi^{\sU}(u)$$ at a point $u \in \rho^{-1}([A])$ and $\widetilde{u} = \varphi(u)$ satisfies $\varphi^{\sK} \circ \rho|_{(u/\sU)_k} = \widetilde{\rho} \circ \varphi_u$. By Lemma \ref{l.GLS}, we can  extend $\varphi_u$ and $\varphi^{\sK}$ to a biholomorphic map
 $$\Phi^{\rho}: (u/\sU)_{\sO} \to (\widetilde{u}/\widetilde{\sU})_{\sO} \mbox{ and } \Psi^{\sK}: ([A]/\sK)_{\sO} \to ([\widetilde{A}]/\widetilde{\sK})_{\sO}$$
 such that $\Psi^{\sK} \circ \rho|_{(u/\sU)_{\sO}} = \widetilde{\rho} \circ \Phi^{\rho}. $ Then $\Phi^{\rho}_*: \sR((u/\sU)_{\sO}) \to \sR((\widetilde{u}/\widetilde{\sU})_{\sO})$ sends $(\sP^{\rho}_u/ \sP^{\rho})_{k-1}$
 to $(\widetilde{\sP}^{\widetilde{\rho}}_{\widetilde{u}}/\widetilde{\sP}^{\widetilde{\rho}})_{k-1}.$
 But the latter property is independent of the choice of the extension $\Phi^{\rho}$ by Lemma \ref{l.coordi}. It follows that $\varphi_u$ sends $(\sP^{\rho}_u/ \sP^{\rho})_{k-1}$
 to $(\widetilde{\sP}^{\widetilde{\rho}}_{\widetilde{u}}/\widetilde{\sP}^{\widetilde{\rho}})_{k-1}.$

 Similarly, applying Lemma \ref{l.GLS} with $\varphi \circ \mu|_{(u/\sU)_k} = \widetilde{\mu} \circ \varphi_u,$ we can find biholomorphic maps
 $$\Phi^{\mu}: (u/\sU)_{\sO} \to (\widetilde{u}/\widetilde{\sU})_{\sO} \mbox{ and } \Psi^{X}: (x/X)_{\sO} \to (\widetilde{x}/\widetilde{X})_{\sO}$$ with $x = \mu(u)$ and $ \widetilde{x} = \widetilde{\mu}(\widetilde{u})$
 such that $\Psi^{X} \circ \mu|_{(u/\sU)_{\sO}} = \widetilde{\mu} \circ \Phi^{\mu}. $
 By the same argument as before, we see that $\varphi_u$ sends $(\sP^{\mu}_u/ \sP^{\mu})_{k-1}$
 to $(\widetilde{\sP}^{\widetilde{\mu}}_{\widetilde{u}}/\widetilde{\sP}^{\widetilde{\mu}})_{k-1}.$

 From $\sP = \sP^{\rho} \cap \sP^{\mu}$ and $\widetilde{\sP} = \sP^{\widetilde{\rho}} \cap \sP^{\widetilde{\mu}}$, we see that $\varphi_u$ sends $(\sP_u/ \sP)_{k-1}$
 to $(\widetilde{\sP}_{\widetilde{u}}/\widetilde{\sP})_{k-1}.$
 Thus the G-structures $\sP$ and $\widetilde{\sP}$ are locally equivalent up to order $k_o$ in the sense of Definition \ref{d.G} (2) with $U = \rho^{-1}(\sW)$. By Theorem \ref{t.G}, they are locally equivalent in the sense of Definition \ref{d.G} (1). Hence we can find open subsets $O \subset \sU$ and $\widetilde{O} \subset \widetilde{\sU}$ with a biholomorphic map $\Phi$ satisfying the conditions of Lemma \ref{l.Hw}. The holomorphic map $F^{\flat}$ in Lemma \ref{l.Hw} (iv) implies that $\sK$ and $\widetilde{\sK}$ are germ-equivalent.
\end{proof}

\begin{remark}\label{r.JeVe}  Theorem \ref{t.family} combined with Theorem \ref{t.example} and Lemma \ref{l.Hw10}  implies the following  version of Theorem \ref{t.JeVe}.
\begin{itemize} \item
 Let $S \subset \BP^n$ be a submanifold. Then there exists  a positive integer $\ell$ (depending on $S$) such that
if $S$ and a submanifold $\widetilde{S} \subset \BP^n$ have  contact of order $\ell$, then  there exists $\gamma \in {\rm PGL}(n+1)$ such that $S \cap \gamma(\widetilde{S})$ contains a nonempty open subset  in $S$.
\end{itemize}
This, however, is weaker than Theorem \ref{t.JeVe} as the number $\ell$
from Theorem \ref{t.family} depends on $S$, while the number $\ell$ in Theorem \ref{t.JeVe} depends only on the dimension $n$. We do not know if the number $\ell$ in Theorem \ref{t.family} (resp. the number $k_o$ in Theorem \ref{t.G})
can be bounded in terms of the dimension of the universal family $\sU$ (resp. the dimension of $M$).
\end{remark}

Finally, to derive Theorem \ref{t.CF} from Theorem \ref{t.family}, we recall the following version of Cartan-Fubini type extension theorem. This is equal to  Theorem 3.9 of \cite{Hw19}: the only difference is that we are making  the additional assumption (ii) below to simplify the statement. As explained in the proof of Theorem 3.9 in \cite{Hw19}, Theorem \ref{t.CFO} was  proved essentially in \cite{HM01}.

\begin{theorem}\label{t.CFO}
Let $X$ and $\widetilde{X}$ be Fano manifolds of Picard number 1.
 Let $\sM$ (resp. $\widetilde{\sM}$) be a family of minimal rational curves on $X$ (resp. $\widetilde{X}$). Assume \begin{itemize}
  \item[(i)] the subschemes $\sM_x$ and $\widetilde{\sM}_{\widetilde{x}}$ are irreducible for  general $x \in X$ and $\widetilde{x} \in \widetilde{X}$; and
      \item[(ii)]  general members of $\sM$ and $\widetilde{\sM}$ are nonsingular. \end{itemize}
          If there are members $A \subset X$ of $\sM$ and $\widetilde{A} \subset \widetilde{X}$ of $\widetilde{\sM}$ with a biholomorphic map  $$\Phi: (A/X)_{\sO} \cong (\widetilde{A}/\widetilde{X})_{\sO},$$ then there exists a biregular morphism $ X \to \widetilde{X}$ which induces $\Phi$. \end{theorem}

\begin{proof}[Proof of Theorem \ref{t.CF}]
 To use Theorem \ref{t.family} to deduce Theorem \ref{t.CF} from Theorem \ref{t.CFO},  it suffices to show that general members of $\sM$ and $\widetilde{\sM}$ in Theorem \ref{t.CF} are free  and their normal bundles have the separating property.
 That general members are free is well-known (e.g. p. 355 of \cite{HM98}). Since the normal bundle of a free rational curve is a direct sum of nonnegative line bundles, if the normal bundle does not have the separating property, then it has to be a trivial bundle. But this contradicts the assumption that $\sM_x$ is irreducible, e.g. from Proposition 1.2.2 of \cite{HM98}. \end{proof}

\begin{remark}
The original formulation of Cartan-Fubini type extension theorem, Main Theorem in \cite{HM01}, is stronger than
Theorem \ref{t.CFO}. It is in terms of the local equivalence of geometric structures determined by the varieties of minimal rational tangents in Definition \ref{d.vmrt} and Theorem \ref{t.CFO} is one, albeit important, step in its proof.
Thus Theorem \ref{t.CF} does not give a purely algebraic formulation of the original version of Cartan-Fubini type extension theorem. In this regard, it would be very interesting to study the geometric structures determined by the varieties of minimal rational tangents from the perspective of Section \ref{s.Gstr}.
\end{remark}

\medskip
{\bf Acknowledgment}.
I would like to thank Ciro Ciliberto, Tohru Morimoto and Shigeru Mukai for valuable discussions. I am very grateful to Yong Hu for pointing out a gap in our approach to Question \ref{q.K3} in an earlier version of the paper and the referees for valuable suggestions.


\begin{thebibliography}{KSWZ}

\bibitem{BM} Beauville, A. and Merindol, J.-Y. : Sections hyperplanes des surfaces K3.
Duke Math. J. {\bf 55} (1987) 873--878


\bibitem{CS} C\v{a}p, A. and  Slov\'{a}k, J.: {\em Parabolic Geometries I. Background and general theory}. Mathematical Surveys and Monographs {\bf 154}, American Mathematical Society, Providence, RI, 2009

\bibitem{CDS} Ciliberto, C., Dedieu, T. and Sernesi, E.: Wahl maps and extensions of canonical curves
and K3 surfaces. J. reine angew. Math. {\bf 761} (2020) 219--245



\bibitem{CLM} Ciliberto, C., Lopez, A. and Miranda, R.: Projective degenerations of K3 surfaces, Gaussian maps, and Fano threefolds. Invent. Math. {\bf 114} (1993) 641--667

\bibitem{GLS} Greuel, G.-M., Lossen, C. and Shustin, E.: {\em Introduction to singularities and deformations}. Springer Verlag, Berlin Heidelberg, 2007

\bibitem{Gr66} Griffiths, P.: The extension problem in complex analysis II; Embeddings with positive normal bundle. Amer. J. Math.
{\bf 88} (1966) 366--446


\bibitem{Gr74} Griffiths, P.: On Cartan's method of Lie groups and moving frames as applied to uniqueness and existence questions in differential geometry. Duke Math. J. {\bf 41} (1974) 775--814

\bibitem{Ha70} Hartshorne, R.: {\em Ample subvarieties of algebraic varieties}. Lecture Notes in Math. {\bf 156} Springer Verlag, Berlin-New York, 1970

\bibitem{Ha10} Hartshorne, R.: {\em Deformation theory}.  Grad. Texts in Math. {\bf 257}, Springer Verlag, New York, 2010


\bibitem{Hir}
Hirschowitz, A.:
On the convergence of formal equivalence between embeddings.
Ann. of Math. {\bf 113} (1981) 501--514

\bibitem{Huy} Huybrecht, D.: {\em Lectures on K3 surfaces.} Cambridge Studies in Advanced Mathematics {\bf 158}, Cambridge Univ. Press, Cambridge, 2016.

\bibitem{Hw01} Hwang, J.-M.: Geometry of minimal rational curves on Fano manifolds. in {\em School on Vanishing Theorems and Effective Results in Algebraic Geometry. Trieste, 2000},  ICTP Lect. Notes, {\bf 6}, Abdus Salam Int. Cent. Theoret. Phys., Trieste, 2001, pp. 335--393

\bibitem{Hw10} Hwang, J.-M. : Equivalence problem for minimal rational curves with isotrivial varieties of minimal rational tangents. Ann. Sci. Ec. Norm. Super.  {\bf 43} (2010) 607--620

\bibitem{Hw19} Hwang, J.-M.: An application of Cartan's equivalence method to Hirschowitz's conjecture on the formal principle. Ann. Math.   {\bf 189} (2019) 979--1000



\bibitem{HM98} Hwang, J.-M., Mok, N.: Varieties of minimal rational tangents on uniruled projective manifolds. in {\em Several complex variables.} MSRI
Publications {\bf 37}, 351--389,  Cambridge University Press, Cambridge, 1999


\bibitem{HM01}  Hwang, J.-M. and  Mok, N.: Cartan-Fubini type extension of holomorphic maps for Fano manfiolds of Picard number 1. J. Math. Pures Appl. (9) {\bf 80} (2001) 563--575

\bibitem{Jen} Jensen, G.: {\em Higher order contact of submanifolds of homogenoues spaces}. Springer Lect. Notes Math. {\bf 610}, Springer Verlag, Berlin-New York, 1977

 \bibitem{Mor}
    Morimoto, T.:  Sur le probl\`eme d'\'equivalence des structures g\'eom\'etriques. Japan. J. Math. {\bf 9} (1983) 293--372

\bibitem{MR} Morrow, J. and Rossi, H.: Some general results on equivalence of embeddings. {\em Recent developments in several complex variables.} Ann. Math. Studies {\bf 100},  299--325, Princeton Univ. Press, Princeton, 1981

\bibitem{NS} Nirenberg, L. and Spencer, D.:   On rigidity of holomorphic imbeddings. {\em Contributions to function theory},  133--137, Tata Institute of Fundamental Research, Bombay, 1960



\bibitem{Pa} Palamodov, V. P.: Moduli in versal deformations of complex spaces.  {\em Vari\'et\'es analytiques compactes}. Springer Lect. Notes Math. {\bf 683}, Springer Verlag, Berlin-Heidelberg, 1978




\bibitem{Ve} Verderesi,  J. A.: Contact et congruence de sous vari\'et\'es. Duke Math. J. {\bf 49} (1982) 513--515

\end{thebibliography}
\end{document}